\documentclass[11pt,a4paper]{article}
\usepackage[T1]{fontenc}
\usepackage[utf8]{inputenc}
\usepackage[english]{babel}
\usepackage{lmodern,microtype,amsmath,amssymb,amsthm,mathtools}
\usepackage[margin=28mm]{geometry}
\usepackage{enumitem,booktabs,array,xcolor,tikz}
\usepackage{needspace}
\usetikzlibrary{arrows.meta,positioning}
\usepackage[colorlinks=true,linkcolor=blue!45!black,citecolor=blue!45!black,urlcolor=blue!45!black]{hyperref}
\newtheorem{theorem}{Theorem}[section]
\newtheorem{proposition}[theorem]{Proposition}
\newtheorem{lemma}[theorem]{Lemma}
\newtheorem{corollary}[theorem]{Corollary}
\theoremstyle{definition}
\newtheorem{definition}[theorem]{Definition}
\newtheorem{example}[theorem]{Example}

\newcommand{\F}{\mathbb F}
\newcommand{\M}{\mathcal M}
\newcommand{\Mh}{\mathcal M_{\F}^{h}}

\newcommand{\Sink}{\operatorname{Sink}}
\newcommand{\Source}{\operatorname{Source}}

\newcommand{\reach}{\rightsquigarrow}
\newcommand{\lk}{\operatorname{lk}}
\newcommand{\doi}[1]{\href{https://doi.org/#1}{doi:\nolinkurl{#1}}}
\setlist{itemsep=3pt,topsep=5pt}
\title{Morse--Maxwell Theory for Directed Acyclic Graphs}
\author{Oleg R. Musin}
\date{}
\hypersetup{pdftitle={Morse-Maxwell Theory for Directed Acyclic Graphs},
pdfauthor={Oleg R. Musin}}
\begin{document}
\maketitle

\begin{abstract}
The Maxwell graph of a directed acyclic graph records the extrema
of its iterated basin and hill boundaries. Their unweighted counts
need not satisfy Morse inequalities, and the reduction can change
homology. We assign multiplicities using the reduced homology of
lower and upper order links and retain every vertex with nonzero
local homology. This modified reduction preserves both hierarchies,
the local groups, and the homology of all sublevel and superlevel
order complexes. Its Morse numbers satisfy the weak and strong
Morse inequalities over any field. We give explicit formulas in
low dimensions and for complete layered graphs. We also construct
local Morse numbers for path homology and prove agreement with the
reachability construction on complete layered graphs, including
those without transitive arrows. Counterexamples distinguish the
homology theories and show that local Morse degrees need not equal
vertex depths. A comparison with discrete-gradient methods explains
the role of auxiliary choices.
\end{abstract}

\section{Introduction}
Morse theory relates the topology of a space to the changes that occur
as a height level passes its critical points. For a landscape, minima
create components, saddles can merge components or create loops, and
maxima can close holes. On a directed graph, the arrows supply a
notion of descent, but a vertex need not have just one of these roles.
Several branches can merge at once, and a single vertex can contribute
in more than one homological degree.

We start with a finite directed acyclic graph, or DAG. Equivalently,
one can assign distinct heights to its vertices so that every arrow
points downwards. Following all descending paths leads from a vertex
to one or several sinks, which play the role of local minima.
The vertices that reach at least two sinks form the \emph{basin
boundary}. Repeating this operation gives a hierarchy; reversing the
arrows gives the corresponding hierarchy of hill boundaries.
The \emph{Maxwell graph} retains the extrema at every stage and the
reachability relations between them.

The geometric origin of this construction goes back to Maxwell's
\emph{On Hills and Dales} \cite{Maxwell}, published in 1870,
55 years before Morse's paper on critical-point inequalities
\cite{Morse1925}. Maxwell partitions a surface into dales whose
descending slope lines reach the same bottom, and independently into
hills whose ascending lines reach the same summit. Watersheds and
watercourses separate these regions and connect their critical points.
Thus the surface--boundary--extremum structure underlying our
two-dimensional construction is already present in Maxwell's work.
The iterative boundary definition below extends this geometric idea
to arbitrary finite DAGs; it is not a definition stated by Maxwell
in modern graph-theoretic language.

Maxwell also follows the changes in regions as a level rises:
bottoms create regions of depression, bars merge them, and passes
separate regions of elevation. He explicitly assigns multiplicity
$k-1$ when $k$ regions merge at a singular point
\cite[pp.~235--239 of the reprint]{Maxwell}. This anticipates the
local contribution $\dim\widetilde H_0=k-1$ used here. Our extension
allows local links with homology in several degrees, which cannot
be described solely by the number of merging regions.
Related graph constructions appear in
\cite{Musin1992,Musin1993}. The hierarchy records how descending paths
branch, but its level counts do not themselves satisfy Morse
inequalities. For example, a vertex pointing to three sinks joins
three components at once: its contribution has multiplicity two.
Moreover, the original Maxwell reduction can remove a vertex needed
to preserve homology. The tree in Figure~\ref{fig:tree} illustrates
this second obstruction directly.

To formulate a Morse theorem, we must specify the space whose
topology the graph describes. We use its \emph{reachability order}:
$v\le u$ means that a directed path leads from $u$ to $v$.
Its order complex $K(D)$ contains a simplex for each chain of
comparable vertices. Thus three vertices occurring in descending
order determine a filled triangle, even if the graph has no arrow
joining the first to the last. This distinction allows a DAG to
describe topology in dimensions higher than one.

Fix a field $\F$. For a vertex $v$, let $L^-(v)$ be the order complex
of its strict descendants. Adding $v$ and all simplices with greatest
vertex $v$ attaches a cone on $L^-(v)$. The reduced homology of this
link therefore measures the possible changes at $v$. Define
\[
 \mu_r^-(v)=\dim_\F\widetilde H_{r-1}(L^-(v);\F),\qquad r\ge0,
\]
using $\widetilde H_{-1}(\varnothing;\F)=\F$.
A sink contributes one generator in degree zero; a lower link with
three contractible components contributes two in degree one.
The resulting \emph{Morse numbers} and Betti numbers are
\begin{equation}\label{eq:counts}
 m_r(D)=\sum_{v\in V(D)}\mu_r^-(v),\qquad
 \beta_r(D)=\dim_\F H_r(K(D);\F).
\end{equation}

\Needspace{9\baselineskip}
\begin{theorem}[Morse inequalities]\label{thm:main}
For every finite directed acyclic graph $D$, every field $\F$, and every
$q\ge0$,
\begin{equation}\label{eq:strong}
 \sum_{r=0}^{q}(-1)^{q-r}m_r(D)
 \ \ge\
 \sum_{r=0}^{q}(-1)^{q-r}\beta_r(D),
\end{equation}
with equality for all sufficiently large $q$.
The weak inequalities $m_r(D)\ge\beta_r(D)$ also hold.
\end{theorem}

Equivalently, the Morse and Poincar\'e polynomials satisfy
$m_D(t)-P_D(t)=(1+t)Q_D(t)$ with nonnegative coefficients in $Q_D$.
Their definitions and interpretation are collected in
Section~\ref{sec:polynomials}.

Reversing all arrows gives the ascending inequalities. We retain
every vertex selected by either Maxwell hierarchy and every vertex
with nonzero lower or upper local homology. The resulting
\emph{homological Maxwell graph} preserves both hierarchies, all
remaining local groups, and the homology of every sublevel and
superlevel order complex. Thus it preserves the data in
\eqref{eq:counts}, together with their evolution through the graph.
The proof uses the standard homological argument behind Morse
inequalities \cite{Milnor,Minian,Skoldberg}; the additional requirement
is to preserve both Maxwell hierarchies under reduction.
The deletion argument is related to reductions of finite posets
\cite{BarmakMinian}.

The homology of $K(D)$ is reachability homology \cite{HR}.
For a DAG it also equals path homology of the transitive closure
\cite{Muranov}. Path homology of the original graph \cite{GLMY},
however, uses its actual arrows and can give different answers.
It is also distinct from the singular simplicial theories studied
in \cite{BWYY}. Section~\ref{sec:outlook} keeps these constructions
separate. It gives Morse inequalities for the unchanged Maxwell
graph using local multiplicities computed there, and for ordinary
path homology using relative local groups at the vertices.
Simply substituting path Betti numbers into \eqref{eq:strong} fails:
a subdivided diamond has $m_1=0$ but first path Betti number one.
In contrast, complete layered graphs give a substantial family
where all three constructions agree, even without transitive arrows.

The distinction concerns both the space being measured and the
data retained by a reduction. Reachability homology forgets whether
a comparable pair is joined by an arrow or by a longer directed
path. Ordinary path homology retains that distinction: deleting an
interior vertex from an allowed path may produce a forbidden face,
and linear combinations must cancel all such faces. The unchanged
Maxwell graph makes a different choice, retaining only the vertices
selected by the two boundary hierarchies. Its order complex can lose
homology or acquire new classes. Our modified reduction retains
additional vertices when their local homology requires it, so that
the original reachability homology survives. Two connected examples in Section~\ref{sec:differences} make
these distinctions concrete. The order complex of $\Sigma'$ is
homotopy equivalent to $S^2$, and our reduction preserves its
nonzero second homology. Both the original Maxwell graph and the
path complex instead have the homology of $S^1$. A second graph,
$\Lambda$, has the homology of $S^1$ in all three theories, but its
correction polynomial is $Q=1$ for our construction and $Q=1+t$
for the local path construction. Thus preservation of topology
and sharpness of the inequalities are separate issues, even when
the homology under consideration is nontrivial.

An alternative approach begins with a discrete gradient.
Forman \cite{Forman} pairs cells with incident cells of the next
dimension, leaving critical cells unpaired; the Morse differential
is determined by gradient paths. Lin, Wang, and Yau
\cite{LinWangYau} develop a digraph version from a nonnegative
function on vertices, whose value on a path is the sum of its vertex
values. Under an extension condition and invariance of the path
chain complex under the gradient, their reduction preserves path
homology. This is a choice-dependent reduction of chains.
Our local groups and Morse numbers are fixed by the directed graph;
no auxiliary matching is chosen. Their multiplicities need not be
minimal, and the differential of the resulting Morse complex still
requires choices. Section~\ref{sec:gradient-comparison} explains this
distinction and why an arbitrary Forman matching need not arise from
the vertex functions used in \cite{LinWangYau}.

The relation with smooth Morse theory is through local attachments.
At a nondegenerate critical point of index $r$, the relative homology
of successive sublevel sets has one generator in degree $r$.
Here the corresponding relative group is
$\widetilde H_{r-1}(L^-(v);\F)$; it may have several generators,
or be nonzero in several degrees at the same vertex. This is the
homological viewpoint familiar from Morse theory on posets and
algebraic Morse theory \cite{Minian,Skoldberg}.
Kozlov's \emph{Organized Collapse} develops the approach through
acyclic matchings and simplicial collapses, together with its
algebraic and poset formulations
\cite[Chapters~9--11 and 15--16]{Kozlov}.
These provide a natural framework for comparing reductions of
the order complex of a DAG. Collapsibility implies acyclicity over every field; our vertex deletions
require only acyclicity of local links over the chosen field,
so this condition is weaker. Our specific goal is to combine
the local homological viewpoint with the two Maxwell hierarchies and a reduction that
preserves both lower and upper local data. Applying Forman's theory
directly to $K(D)$ is another valid way to obtain a Morse complex,
but its critical simplices need not be grouped by the original
vertices or retain these hierarchies.

Section~\ref{sec:definitions} gives the definitions and the tree
example. Section~\ref{sec:proofs} proves the reduction theorem and
gives computational bounds. Section~\ref{sec:examples} develops
explicit formulas in low dimensions and for complete layered graphs.
It also shows why a vertex's homological degrees cannot in general
be inferred from its depth. Section~\ref{sec:outlook} compares the
homology theories and ends with three questions suggested by these
results.

\section{Definitions}\label{sec:definitions}
\subsection{Directed graphs and their order complexes}
Throughout, $D=(V,A)$ is a finite directed acyclic graph (DAG),
without loops or parallel arrows. We use $u\reach v$ for
reachability, allowing a path of length zero. A sink has no outgoing
arrow and a source has no incoming arrow. The opposite graph
$D^{\mathrm{op}}$ reverses all arrows.

A \emph{height function} is a function $f:V\to\mathbb R$ strictly
decreasing on arrows. Every DAG admits a height function. Different
choices give the same construction, since only reachability is used.

\begin{definition}
The \emph{reachability order} $P=P(D)$ is given by
$v\le u$ if $u\reach v$. Its \emph{order complex} $K(D)=\Delta P$ has
one simplex for each nonempty chain of $P$.
For $U\subseteq V$, write $P|_U$ for the restricted order and
$K(U)=\Delta(P|_U)$.
\end{definition}

For example, $w<v<u$ determines a triangle in $K(D)$, even if
$u\to w$ is not an arrow of $D$. Thus $K(D)$ is not generally the
one-dimensional underlying graph. It is the clique complex of the
comparability graph: pairwise comparable vertices form a chain.
Throughout the paper homology has
coefficients in an arbitrary fixed field $\F$, for example
$\mathbb Q$, $\mathbb R$, or a finite field $\mathbb F_p$.
All Morse inequalities use the same field for the local groups
and the global Betti numbers. We use ordinary homology for
$K(D)$ and augmented reduced homology for local links. The empty
complex has $\widetilde H_{-1}=\F$ and no other reduced homology.
A complex is \emph{acyclic over $\F$} if all its augmented reduced
homology groups vanish; in particular it is nonempty.

The \emph{dimension} $d(D)=\dim K(D)$ is the maximum length of a
directed path, measured in arrows. The \emph{depth} $s(v)$ is the
maximum length of such a path starting at $v$.
For complexes on disjoint vertex sets, $A*B$ denotes their simplicial
join: its simplices are unions of a simplex from each factor,
with empty faces allowed. In our augmented convention,
$A*\varnothing=A$.

\subsection{Basins, hills, and the Maxwell graph}
For $U\subseteq V$, let $D[U]$ denote the induced directed subgraph.
For any directed graph $G$, we write
\[
 \begin{aligned}
 \Sink(G)&=\{v\in V(G):\text{no arrow of $G$ starts at $v$}\},\\
 \Source(G)&=\{v\in V(G):\text{no arrow of $G$ ends at $v$}\}.
 \end{aligned}
\]
These are its sets of sinks and sources. An isolated vertex belongs
to both sets. For an induced subgraph, these conditions are tested
using only its own arrows; a vertex can therefore become a sink or
source after other vertices have been removed.

\begin{definition}
The \emph{basin} and \emph{hill} of a vertex $v$ are
\[
 B_D(v)=D[\{u:u\reach v\}],\qquad
 H_D(v)=D[\{u:v\reach u\}].
\]
The sets of reachable sinks and preceding sources are
\[
 Z_D(v)=\{s\in\Sink(D):v\reach s\},\qquad
 T_D(v)=\{t\in\Source(D):t\reach v\}.
\]
The \emph{basin boundary} and \emph{hill boundary} are
\[
 \Gamma_B(D)=D[\{v:|Z_D(v)|\ge2\}],\qquad
 \Gamma_H(D)=D[\{v:|T_D(v)|\ge2\}].
\]
\end{definition}

Thus the basin boundary consists of vertices from which descending
paths lead to different minima. Basins may overlap. The word
``boundary'' refers to this combinatorial definition.

\begin{definition}
The descending and ascending boundary hierarchies are
\[
 \begin{array}{lll}
 D_B^0=D,&D_B^{j+1}=\Gamma_B(D_B^j),& B_j=\Sink(D_B^j),\\
 D_H^0=D,&D_H^{j+1}=\Gamma_H(D_H^j),& H_j=\Source(D_H^j).
 \end{array}
\]
The boundary of the empty graph is empty. The \emph{topographically
essential set} is
\[
 W(D)=\bigcup_{j\ge0}B_j\ \cup\ \bigcup_{j\ge0}H_j.
\]
\end{definition}

At each step we use the extrema of the current stage. Since a nonempty
stage loses its extrema at the next step, both hierarchies terminate.
A vertex may belong to both a descending and an ascending level.

\begin{definition}
For $U\subseteq V$, let $D_U$ be the downward Hasse diagram of $P(D)|_U$.
There is an arrow $u\to v$ exactly when $v<u$ and no element of $U$
lies strictly between them. The \emph{Maxwell graph} is
\[
 \M(D)=D_{W(D)}.
\]
\end{definition}

An arrow of $D_U$ can represent a path through omitted vertices.
The Hasse diagram is the transitive reduction of the restricted
comparability graph \cite{Aho}.

\subsection{Local Morse data}
The basin and hill are induced subgraphs, whereas the local objects
below are simplicial complexes built from reachability. To relate them,
let $P_{<v}$ and $P_{>v}$ be the strict lower and upper sets of $v$.
As vertex sets,
\[
 P_{<v}=V(H_D(v))\setminus\{v\},\qquad
 P_{>v}=V(B_D(v))\setminus\{v\}.
\]
We equip these sets with the order inherited from $P(D)$. Thus the
lower set consists of all strict descendants, and the upper set of
all strict ancestors, not only immediate neighbours. Applying $\Delta$
turns each nonempty chain into a simplex: two comparable vertices
span an edge, three vertices in a chain span a filled triangle, and
so on. For example, if the hill is the directed path
$v\to a\to b\to c$, then $P_{<v}$ is the chain $c<b<a$,
and $\Delta(P_{<v})$ is a filled triangle, although the graph on
$a,b,c$ has only two arrows. Reachability supplies the comparison
between $a$ and $c$.

The notation $\widetilde H_i(L;\F)$ denotes reduced homology.
These groups are vector spaces over $\F$. For $i\ge1$ they equal
ordinary homology; if a nonempty complex $L$ has $c$ connected
components, then $H_0(L;\F)\cong\F^c$ but
$\widetilde H_0(L;\F)\cong\F^{c-1}$. In particular, a connected
complex has zero reduced homology in degree zero. We use the
augmented convention $\widetilde H_{-1}(\varnothing;\F)=\F$;
this group is zero for every nonempty complex.

\begin{definition}\label{def:local}
The lower and upper order links, local groups, and multiplicities are
\[
 L^-(v)=\Delta(P_{<v}),\qquad L^+(v)=\Delta(P_{>v}),
\]
\[
 E_r^\pm(v)=\widetilde H_{r-1}(L^\pm(v);\F),\qquad
 \mu_r^\pm(v)=\dim_\F E_r^\pm(v)\quad(r\ge0).
\]
A vertex is \emph{descending critical} if some $\mu_r^-(v)$ is nonzero,
and \emph{ascending critical} if some $\mu_r^+(v)$ is nonzero.
Otherwise it is regular in the corresponding direction.
The descending and ascending critical sets are
\[
 \begin{aligned}
 C^-(D)&=\{v\in V:\mu_r^-(v)>0\text{ for some }r\ge0\},\\
 C^+(D)&=\{v\in V:\mu_r^+(v)>0\text{ for some }r\ge0\}.
 \end{aligned}
\]
\end{definition}

Thus the lower and upper order links are the order complexes of the
hill and basin with $v$ removed. Removing $v$ is essential: the order
complex of the whole hill or basin is a cone with apex $v$, hence
has zero reduced homology. The punctured complex can have nontrivial
homology, which records the local contribution of $v$.

The shift from degree $r-1$ to degree $r$ reflects the attachment.
When vertices are added from lower to higher in the order, adding
$v$ attaches the cone $v*L^-(v)$ along $L^-(v)$, and
\[
 H_r(v*L^-(v),L^-(v);\F)
 \cong\widetilde H_{r-1}(L^-(v);\F).
\]
For a sink the lower link is empty, so $\mu_0^-=1$. Two isolated
vertices in the lower link give $\mu_1^-=1$; a link homotopy
equivalent to a circle gives $\mu_2^-=1$. In each of these examples
all other lower multiplicities vanish. A nonempty acyclic lower
link gives zero in every degree. More generally, one vertex can
contribute several generators, possibly in different degrees.
Regularity requires every local group to vanish; a zero alternating
sum of their dimensions is insufficient. The degree $r$ of a local
generator is its Morse degree, which need not agree with the
boundary level of the vertex.

\subsection{The homological Maxwell graph}\label{sec:tree-example}
The boundary hierarchies can miss a vertex whose local homology is
nonzero. To retain both the hierarchy and the homological data, we
adjoin all descending and ascending critical vertices to $W(D)$.

\begin{definition}\label{def:homological}
Let $W(D)=\bigcup_{j\ge0}(B_j\cup H_j)$ be the set of extrema
in the two boundary hierarchies, and let $C^-(D),C^+(D)$ be the
critical sets of Definition~\ref{def:local}. The \emph{homological essential set}
and \emph{homological Maxwell graph} are
\[
 S_\F(D)=W(D)\cup C^-(D)\cup C^+(D),\qquad
 \Mh(D)=D_{S_\F(D)}.
\]
The graph retains the labels $B_j,H_j$ and the local groups
$E_r^\pm(v)$ at its vertices.
\end{definition}

The construction keeps every vertex selected by the original Maxwell
hierarchies and adds any vertex with nonzero local homology in either
direction. Thus a vertex is omitted exactly when it lies outside
$W(D)$ and both its lower and upper order links are acyclic over $\F$.
Some vertices in $W(D)$ may be regular in both directions; they are
retained to preserve the hierarchy. The additional critical vertices,
and hence $S_\F(D)$, can depend on the coefficient field.

The arrows are reconstructed from the original reachability order.
For $u,v\in S_\F(D)$, there is an arrow $u\to v$ precisely when
$v<u$ and no retained vertex lies strictly between them. Consequently,
a new arrow may replace a path through omitted vertices; simply
taking the induced subgraph on $S_\F(D)$ would not suffice.
The labels specify which retained vertices belong to each $B_j,H_j$,
while the groups $E_r^\pm(v)$ record their local Morse contributions.

\medskip\noindent\textbf{Example: a vertex missed by the hierarchy.}
Consider the nine-vertex directed tree $T$ with arrows
\[
 a\to p,\ b\to p,\ p\to x,\ c\to x,\
 x\to q,\ x\to d,\ q\to e,\ q\to f.
\]
Its only nonempty levels are
\[
 B_0=\{d,e,f\},\quad B_1=\{q\},\qquad
 H_0=\{a,b,c\},\quad H_1=\{p\}.
\]
Thus $W(T)=V(T)\setminus\{x\}$, and the original Maxwell graph
deletes $x$. The lower link at $x$ is the path $e-q-f$ together
with the isolated vertex $d$; the upper link is the path $a-p-b$
together with the isolated vertex $c$. Each has two contractible
components, so
\[
 \widetilde H_0(L^-(x);\F)\cong
 \widetilde H_0(L^+(x);\F)\cong\F.
\]
Hence $\mu_1^-(x)=\mu_1^+(x)=1$: the vertex belongs to both critical
sets, although neither boundary hierarchy selects it. The homological
construction therefore retains it, giving
\[
 S_\F(T)=V(T),\qquad \Mh(T)=T.
\]
Figure~\ref{fig:tree} compares the original tree with its ordinary
Maxwell graph.

\begin{figure}[htbp]
\centering
\begin{tikzpicture}[>=Stealth,scale=1.08,every node/.style={font=\normalsize}]
\node (a) at (-1,3) {$a$};\node (b) at (0,3) {$b$};
\node (c) at (1.4,3) {$c$};\node (p) at (-.5,2) {$p$};
\node[draw,circle,inner sep=3pt] (x) at (.3,1) {$x$};
\node (q) at (-.5,0) {$q$};\node (d) at (1.4,0) {$d$};
\node (e) at (-1,-1) {$e$};\node (f) at (0,-1) {$f$};
\foreach \s/\t in {a/p,b/p,p/x,c/x,x/q,x/d,q/e,q/f}
 \draw[->] (\s)--(\t);
\draw[->,thick] (2.2,.8)--(3.5,.8);\node at (2.85,1.2) {$\M$};
\node (aa) at (4.2,3) {$a$};\node (bb) at (5.2,3) {$b$};
\node (cc) at (6.6,2) {$c$};\node (pp) at (4.7,2) {$p$};
\node (qq) at (4.7,0) {$q$};\node (dd) at (6.6,0) {$d$};
\node (ee) at (4.2,-1) {$e$};\node (ff) at (5.2,-1) {$f$};
\foreach \s/\t in {aa/pp,bb/pp,pp/qq,pp/dd,cc/qq,cc/dd,qq/ee,qq/ff}
 \draw[->] (\s)--(\t);
\end{tikzpicture}
\caption{The original reduction removes $x$ and changes reachability
homology. The homological Maxwell graph retains $x$ and equals the
tree on the left.}
\label{fig:tree}
\end{figure}
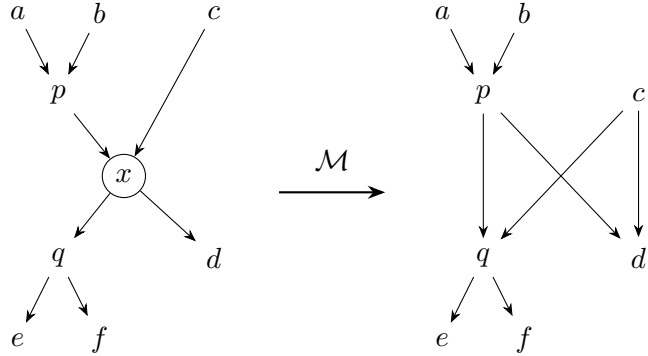

Why does retaining $x$ matter? It is comparable with every vertex of
$T$, so $K(T)$ is a cone and is contractible. After deleting $x$,
every upper-link vertex is above every lower-link vertex. Thus
\[
 K(\M(T))=L^+(x)*L^-(x)\simeq S^0*S^0\cong S^1.
\]
The ordinary reduction changes the first Betti number from zero to
one; the homological Maxwell graph keeps the tree intact.

Proposition~\ref{prop:preservation} proves that this reduction preserves
the homology of the order complex and that the local groups at retained
vertices can be recomputed in $P(D)|_{S_\F(D)}$. Omitted vertices
have zero multiplicities in both directions, so both sets of Morse
numbers are preserved as well.

\subsection{Morse and Poincar\'e polynomials}\label{sec:polynomials}
The descending Morse numbers are $m_r(D)$ from \eqref{eq:counts}.
The ascending ones are
\[
 m_r^+(D)=\sum_{v\in V}\mu_r^+(v)=m_r(D^{\mathrm{op}}).
\]
The order complexes of $P$ and its opposite have the same simplices,
so they have the same Betti numbers.

We collect the notation used in the examples:
\[
 \begin{array}{ll}
 m_D(t)=\displaystyle\sum_{r\ge0}m_r(D)t^r
   &\text{Morse polynomial},\\[3pt]
 P_D(t)=\displaystyle\sum_{r\ge0}\beta_r(D)t^r
   &\text{Poincar\'e polynomial}.
 \end{array}
\]
Thus $P(D)$ is a poset, while $P_D(t)$ is its homological generating
polynomial. The polynomial $Q_D(t)$ is determined by
\begin{equation}\label{eq:polynomial}
 m_D(t)-P_D(t)=(1+t)Q_D(t).
\end{equation}
Theorem~\ref{thm:main} implies that $Q_D$ is a polynomial with
nonnegative integer coefficients; a direct proof is given in
Corollary~\ref{cor:polynomial}. For $q\ge0$,
\[
 [t^q]Q_D(t)=
 \sum_{r=0}^q(-1)^{q-r}\bigl(m_r(D)-\beta_r(D)\bigr).
\]
Thus its coefficients record the differences in the strong Morse
inequalities, and $Q_D=0$ exactly when every Morse number equals
the corresponding Betti number. Also
\[
 \sum_r\bigl(m_r(D)-\beta_r(D)\bigr)=2Q_D(1).
\]
For the ascending data, write $m_D^+(t)=\sum_r m_r^+(D)t^r$ and
$m_D^+(t)-P_D(t)=(1+t)Q_D^+(t)$.
Both $m_D$ and $Q_D$ may depend on the field, as does $P_D$.

For the tree $T$ in Section~\ref{sec:tree-example}, the nonzero
descending contributions are
\[
 \mu_0^-(d)=\mu_0^-(e)=\mu_0^-(f)=1,\qquad
 \mu_1^-(q)=\mu_1^-(x)=1.
\]
Since $K(T)$ is contractible, this gives
\[
 m_T(t)=3+2t,\qquad P_T(t)=1,\qquad Q_T(t)=2.
\]
Thus $m_T(t)-P_T(t)=2(1+t)$, illustrating the Morse identity.

\section{Proofs}\label{sec:proofs}
\subsection{The Morse inequalities}
Choose a linear extension
$v_1,\ldots,v_n$ of $P(D)$ in which strict lower elements appear first.
Set
\[
 K_i=K(\{v_1,\ldots,v_i\}),\qquad K_0=\varnothing.
\]
At step $i$ we attach all chains whose greatest element is $v_i$.

\begin{lemma}[Local attachment]\label{lem:attachment}
For every $r\ge0$,
\[
 H_r(K_i,K_{i-1})\cong E_r^-(v_i).
\]
\end{lemma}
\begin{proof}
The new closed star is the cone $v_i*L^-(v_i)$, and its intersection
with $K_{i-1}$ is $L^-(v_i)$. Therefore its relative chains are the
augmented chains of $L^-(v_i)$ shifted up by one degree.
This also covers a sink: the empty lower link supplies one relative
generator in degree zero.
\end{proof}

\begin{proof}[Proof of Theorem~\ref{thm:main}]
Let $a_{i,r}$ be the rank of the connecting map
\[
 H_r(K_i,K_{i-1})\longrightarrow H_{r-1}(K_{i-1}).
\]
Here $a_{i,0}=0$ (ordinary $H_{-1}=0$). Set $b_r(K_i)=\dim H_r(K_i)$.
Exactness gives
\[
 b_r(K_i)=b_r(K_{i-1})+\mu_r^-(v_i)-a_{i,r}-a_{i,r+1}.
\]
Summing over the vertices, with $A_r=\sum_i a_{i,r}$, gives
\begin{equation}\label{eq:defect}
 m_r(D)-\beta_r(D)=A_r+A_{r+1},\qquad A_0=0.
\end{equation}
Since each $A_r$ is a sum of ranks, the weak inequalities follow.
Alternating sums cancel $A_1,\ldots,A_q$, while $A_0=0$, giving
\[
 \sum_{r=0}^q(-1)^{q-r}\bigl(m_r(D)-\beta_r(D)\bigr)=A_{q+1}\ge0.
\]
All groups vanish in high degrees, giving eventual equality.
\end{proof}

\begin{corollary}\label{cor:polynomial}
The polynomial $Q_D$ defined by \eqref{eq:polynomial} has nonnegative
integer coefficients. In particular,
$\sum_r(-1)^r m_r(D)=\chi(K(D))$.
\end{corollary}
\begin{proof}
Take $Q_D(t)=\sum_{r\ge0}A_{r+1}t^r$ in \eqref{eq:defect},
and then evaluate at $t=-1$.
\end{proof}

\subsection{Preservation under the Maxwell reduction}
We first prove that retaining $W(D)$ preserves both boundary
hierarchies, and then show that regular vertices can be deleted
without changing the local homology groups.

\begin{lemma}[The boundary hierarchies]\label{lem:hierarchy}
If $W(D)\subseteq U\subseteq V(D)$, then
\[
 B_j(D_U)=B_j(D),\qquad H_j(D_U)=H_j(D)
\]
for every $j$. The stage vertex sets are
$U\cap V(D_B^j)$ and $U\cap V(D_H^j)$, respectively.
\end{lemma}
\begin{proof}
The basin boundary is predecessor-closed: if $u\reach v$ and $v$
reaches two sinks, then $u$ reaches them as well. Repetition shows that
each $D_B^j$ is predecessor-closed in $D$. Consequently every directed
path ending in that stage remains in it, so its internal reachability
is the restricted original order.

Every comparable pair in a finite poset is joined by a chain of covers.
Hence $D_U$ realizes exactly the order $P(D)|_U$. Induct on $j$.
At stage zero the asserted vertex set is immediate. Suppose it is
known at stage $j$. Every vertex of the original stage reaches a sink
in $B_j$, and all those sinks belong to $W\subseteq U$. The restricted
stage therefore has exactly the same sinks. Each retained vertex
reaches the same stage sinks as before, so the condition of reaching
at least two gives the required vertex set at stage $j+1$.
Reverse arrows for the ascending statement.
\end{proof}

\begin{lemma}[Deletion of a regular vertex]\label{lem:deletion}
Suppose $L^-(z)$ is acyclic over $\F$. Deleting $z$ from a finite
poset $P$ preserves the homology of $\Delta P$ and the reduced homology
of every remaining lower order link. The upper statement holds
when $L^+(z)$ is acyclic over $\F$.
\end{lemma}
\begin{proof}
For a simplicial complex $K$ and vertex $z$,
\begin{equation}\label{eq:relative}
 H_r(K,K-z)\cong
 \widetilde H_{r-1}(\lk_K z).
\end{equation}
This follows by taking the cone on the link modulo its base.
For an order complex,
\[
 \lk_{\Delta P}z=L^-(z)*L^+(z).
\]
Over a field, a join with an acyclic factor is acyclic: the augmented
join formula is
\[
 \widetilde H_k(A*B)\cong
 \bigoplus_{i+j=k-1}
 \widetilde H_i(A;\F)\otimes_\F\widetilde H_j(B;\F).
\]
The indices include $-1$. With our augmented convention,
$A*\varnothing=A$, and the term $j=-1$ gives
$\widetilde H_k(A*\varnothing)\cong\widetilde H_k(A)$.
The relative groups in \eqref{eq:relative} vanish, proving preservation
of the homology of the full complex.

The lower link at a remaining vertex $v$ changes only if $z<v$.
Inside $L^-(v)$, the link of $z$ is
\[
 L^-(z)*\Delta(\{u:z<u<v\}).
\]
Its first factor is acyclic. The same relative argument proves
preservation of the reduced homology of $L^-(v)$.
Reversing the order proves the upper statement.
\end{proof}

\begin{proposition}[Preservation of the data]\label{prop:preservation}
Let $S=S_\F(D)$ and $Q=P(D)|_S$. Inclusion $\Delta Q\hookrightarrow K(D)$
induces homology isomorphisms. At every retained vertex it also induces
isomorphisms on all lower and upper local groups. The levels $B_j,H_j$
are unchanged.
\end{proposition}
\begin{proof}
Delete the vertices outside $S$ one at a time. They are initially
regular in both directions. At each step, Lemma~\ref{lem:deletion}
preserves the global homology and all remaining local groups.
Thus every vertex still to be deleted is regular in the current
restricted order, and induction gives the homology isomorphisms.
The boundary levels are preserved by
Lemma~\ref{lem:hierarchy}, since $W\subseteq S$.
\end{proof}

\begin{corollary}[Idempotence]\label{cor:idempotence}
The homological Maxwell graph is canonical for $P(D)$ and $\F$, and
\[
 \Mh(\Mh(D))=\Mh(D).
\]
Its Morse numbers and Betti numbers equal those of $D$.
\end{corollary}
\begin{proof}
The vertex set and its Hasse diagram are defined without choices.
Proposition~\ref{prop:preservation} preserves $W$, both critical sets,
and their local groups. Their union in the reduced graph is therefore
exactly $S$. Taking its Hasse diagram again changes nothing.
Vertices outside $S$ contribute zero to the Morse sums.
\end{proof}

An order ideal contains all lower elements of each of its vertices;
an order filter contains all upper elements.

\begin{proposition}[Sublevel and superlevel homology]\label{prop:filtration}
For every order ideal or order filter $U\subseteq P(D)$, inclusion
\[
 K(S\cap U)\hookrightarrow K(U)
\]
induces homology isomorphisms. These isomorphisms commute with inclusions
between nested such sets.
\end{proposition}
\begin{proof}
For an ideal, delete $U\setminus S$ with larger vertices first.
When $z$ is deleted, all of its original lower elements are still
present, so its lower link is the original acyclic $L^-(z)$.
Apply Lemma~\ref{lem:deletion} at each step.
For a filter use smaller vertices first and upper links.
For $U_1\subseteq U_2$, both composites
$K(S\cap U_1)\to K(U_2)$ are the same inclusion. This proves naturality.
\end{proof}

In particular, for a height function $f$ decreasing on arrows,
the sets $\{f\le t\}$ are ideals and $\{f\ge t\}$ are filters.
Hence the reduced and original order complexes have isomorphic
sublevel and superlevel persistent homology. The conclusion is a
homology isomorphism; acyclic links need not give homotopy equivalences.

\subsection{A Morse chain complex}
The local groups form the chain groups of a Morse complex.

\begin{proposition}\label{prop:chain}
There is a differential on
\[
 \mathcal C_r(D)=\bigoplus_{v\in S_\F(D)}E_r^-(v)
\]
making $\mathcal C_*(D)$ chain homotopy equivalent to $C_*(K(D);\F)$.
A nonzero component of the differential from the summand at $v$
to the summand at $u$ requires $u<v$.
\end{proposition}
\begin{proof}
Group simplices according to their greatest vertex. Their boundary
matrix is triangular with respect to the poset labels: a face has
the same greatest vertex or a smaller one. The diagonal block at $v$
is the relative complex of $(v*L^-(v),L^-(v))$.

Over a field, choose bases in each diagonal block separating its
homology representatives from elementary acyclic pairs. A pair
$a,b$ has consecutive degrees, the same vertex label, and a nonzero
coefficient $c$ of $b$ in $da$. It can be cancelled in the full complex.
Replace $b$ by $c^{-1}da$ and subtract suitable multiples of $a$ from
the other generators of its degree to remove their new $b$ coefficient.
The equation $d^2=0$ then splits off this two-term acyclic summand.
This is a chain homotopy equivalence.

Cancellation preserves the order condition. A new matrix entry from
label $w$ to label $u$ can only factor through the cancelled label $v$,
so $u\le v\le w$. In particular no other diagonal block changes.
Cancel all the acyclic pairs. The surviving graded space at $v$ is
$E_*^-(v)$, and the remaining differential strictly lowers its vertex
label. Every vertex outside $S_\F(D)$ has zero local groups, giving
the displayed direct sum.
\end{proof}

The basis choices affect the differential but not the numbers $m_r$.
Writing $d_r:\mathcal C_r\to\mathcal C_{r-1}$ gives
\[
 m_r(D)-\beta_r(D)=\operatorname{rank}d_r+\operatorname{rank}d_{r+1}.
\]
Coefficients may connect vertices separated by several cover arrows.
This differential is therefore not generally the incidence matrix of
the reduced graph viewed as a one-dimensional complex.
Reversing the order gives an ascending Morse complex.

\subsection{First boundary levels and subdivision}
\begin{proposition}\label{prop:firstlevel}
A nonsink reaching exactly one sink has zero descending local groups.
If $v\in B_1$ reaches $k$ sinks of $D$, then
\[
 \mu_1^-(v)=k-1,\qquad \mu_r^-(v)=0\quad(r\ne1).
\]
\end{proposition}
\begin{proof}
In the first case the unique sink is a least element of the lower
order, so its complex is a cone. In the second case every strict
descendant of $v$ lies outside the first basin boundary: otherwise
predecessor closure would contradict $v\in\Sink(\Gamma_B(D))$.
Each descendant therefore reaches exactly one sink. Partition them
by this sink. There are $k$ nonempty parts, each with a least element,
and elements of different parts are incomparable. The lower link is
a disjoint union of $k$ cones.
\end{proof}

The proposition applies within each boundary stage. The groups
computed there may differ from the local groups in the original order.

\begin{proposition}\label{prop:subdivision}
Adding or removing transitive arrows does not change $\Mh(D)$.
Subdividing an arrow into a directed path through new vertices of
indegree and outdegree one preserves the graph, its levels, and its
local multiplicities, with the natural identification of old vertices.
\end{proposition}
\begin{proof}
Transitive arrows do not change reachability. For a single subdivision
$u\to x\to v$, the strict lower set of $x$ has greatest element $v$,
and its strict upper set has least element $u$. Both order links are
cones. Deleting $x$ recovers the original order, and
Lemma~\ref{lem:deletion} preserves all old local groups.

The vertex $x$ reaches the same sinks as $v$. Inductively, in every
descending stage containing $x$, the vertex $v$ is present and their
membership at the next step agrees. Thus $x$ is never a stage sink.
The dual statement uses $u$ and the ascending stages. All old levels
are unchanged, so $x$ belongs to neither $W$ nor $C^-\cup C^+$.
Iterate for a longer subdivision.
\end{proof}

\subsection{Computing the construction}
Let $n=|V|$, $m=|A|$, and let $R$ be the number of nonempty chains
of $P(D)$.

The boundary hierarchy can be computed by propagating at most two
distinct reachable-sink identifiers per vertex in reverse topological
order. One identifier means a unique reachable sink; two suffice to
detect ambiguity. Repeat on the retained boundary, and reverse arrows
for the hill hierarchy. Each stage costs $O(n_j+m_j)$, and there are
at most $n$ stages in each direction. The total is $O(n(n+m))$ time
with $O(n+m)$ working space when stages are stored by exit labels.

Reachability itself can be computed in $O(n(n+m))$ time by a search
from each vertex. Once the order complex is listed, dense boundary-rank
computations give a coarse bound of $O(nR^3)$ field operations for all
lower and upper local groups, treating them one at a time, with
$O(R^2+n^2)$ storage. The cancellation procedure in
Proposition~\ref{prop:chain} takes at most $R/2$ steps, each with
$O(R^2)$ matrix updates, hence $O(R^3)$ field operations.

After $S_\F(D)$ is known, the Hasse diagram on it can be obtained in
$O(|S_\F(D)|(n+m))$ time. For each retained $u$, scan its retained
descendants with larger elements first. Select an uncovered $v$ as a
cover of $u$, then mark $v$ and all its descendants. Every noncover
is marked by a preceding cover, and a cover cannot be marked by a
different retained element without an intermediate vertex. Sharing
the marking set among these searches gives $O(n+m)$ work per $u$.

These bounds describe direct implementations. Since a total order
has $R=2^n-1$, enumerating the order complex can already require
exponential time. Recognizing conical links and using sparse matrices
or cancellations can reduce the computation. Over $\mathbb Q$, the
bounds above count field operations rather than bit operations.

\section{Examples}\label{sec:examples}
We compute the descending Morse numbers; reversing arrows gives the
ascending ones. In each example the identity
$m_D(t)-P_D(t)=(1+t)Q_D(t)$, with $Q_D(t)$ having nonnegative
coefficients, verifies all the Morse inequalities.

\subsection{A three-leaf star}
Let $D$ have arrows $v\to a,v\to b,v\to c$; see
Figure~\ref{fig:starcycle}(a). The sinks $a,b,c$ each contribute
$\mu_0^-=1$. The lower link of $v$ is three isolated points, so
$\mu_1^-(v)=2$. All vertices are retained:
\[
 (m_0,m_1)=(3,2),\qquad(\beta_0,\beta_1)=(1,0).
\]
The inequalities read
\[
 3\ge1,\qquad 2-3=-1=0-1.
\]
Equivalently, $m_D(t)-P_D(t)=2+2t=2(1+t)$.

Take
$\mathcal C_0=\F\langle a,b,c\rangle$ and two generators $e_1,e_2$
at $v$ in degree one, with
\[
 de_1=b-a,\qquad de_2=c-a.
\]
The boundary has rank two. Adding $v$ joins three components into one,
and the multiplicity two counts the two independent mergers.

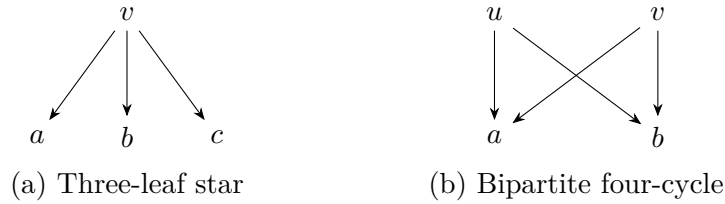
\begin{figure}[htbp]
\centering
\begin{tikzpicture}[>=Stealth,scale=1.08,every node/.style={font=\normalsize}]
\begin{scope}[xshift=0cm]
\node (v) at (0,1.5) {$v$};
\node (a) at (-1.1,0) {$a$};
\node (b) at (0,0) {$b$};
\node (c) at (1.1,0) {$c$};
\foreach \u in {a,b,c}\draw[->] (v)--(\u);
\node at (0,-.6) {(a) Three-leaf star};
\end{scope}
\begin{scope}[xshift=5.5cm]
\node (u) at (-1,1.5) {$u$};\node (w) at (1,1.5) {$v$};
\node (aa) at (-1,0) {$a$};\node (bb) at (1,0) {$b$};
\foreach \s/\t in {u/aa,u/bb,w/aa,w/bb}\draw[->] (\s)--(\t);
\node at (0,-.6) {(b) Bipartite four-cycle};
\end{scope}
\end{tikzpicture}
\caption{Two different roles for degree-one generators: component
mergers in (a), and a merger together with a surviving cycle in (b).
Crossings of arrows are not vertices.}
\label{fig:starcycle}
\end{figure}

\subsection{Bipartite graphs and dimension one}\label{sec:dimone}
A nonempty graph has $d(D)\le1$ precisely when no vertex has
both an incoming and an outgoing arrow. Let $A$ be the set of
nonsinks and $B$ the set of sinks, including isolated vertices.
Every arrow then goes from $A$ to $B$. Put
$p=|A|$, $q=|B|$, let $e$ be the number of arrows, and let $c$ count weakly
connected components.

\begin{proposition}\label{prop:heightone}
For $d(D)\le1$,
\[
 m_D(t)=q+(e-p)t,\qquad
 P_D(t)=c+(e-p-q+c)t,\qquad Q_D(t)=q-c.
\]
In particular, $m_D=P_D$ exactly when every component has one sink.
\end{proposition}
\begin{proof}
The lower link of $a\in A$ is the discrete set of its outgoing
neighbors, so its only contribution is $\deg^+(a)-1$ in degree one.
Each sink contributes one in degree zero. The order complex is the
underlying bipartite graph, whose first Betti number is $e-p-q+c$.
Subtracting gives $(1+t)(q-c)$.
\end{proof}

Thus a forest need not give equality: the three-leaf star has
$Q_D=2$. On $n\ge2$ vertices with $d(D)=1$,
\[
 m_0\le n-1,\qquad
 m_1\le\left\lfloor\frac{(n-1)^2}{4}\right\rfloor .
\]
The first bound is attained by a star. For the second,
$e-p\le p(q-1)=p(n-p-1)$, and a complete bipartite graph with $p$
and $q-1$ as equal as possible attains the bound.

Without fixing $n$, the possible Morse--Betti pairs admit a short
description. They are exactly
\[
 \bigl(q+(q-c+b)t,\ c+bt\bigr),\qquad
 q\ge c\ge1,\quad b\ge0,\quad q=c\ \Longrightarrow\ b=0.
\]
Necessity follows from the proposition: one sink per component
forces a forest. For sufficiency when $q>c$, take a source joined to
$q-c+1$ sinks, add $b$ sources each joined to the same two of these
sinks, and add $c-1$ isolated sinks. If $q=c$ and $b=0$, a single
arrow and $c-1$ isolated vertices suffice.

For the complete bipartite graph with all arrows from $p$ sources
to $q$ sinks,
\[
 m_D(t)=q+p(q-1)t,\qquad
 P_D(t)=1+(p-1)(q-1)t,\qquad Q_D(t)=q-1.
\]
Figure~\ref{fig:starcycle}(b) shows $K_{2,2}$.
Here $(m_0,m_1)=(2,2)$ and $(\beta_0,\beta_1)=(1,1)$.
A Morse complex has generators $e_u,e_v$ in degree one with
$de_u=de_v=b-a$. The difference $e_u-e_v$ represents the surviving
cycle. For $K_{3,4}$ the corresponding numbers are $(4,9)$ and
$(1,6)$. There are only three vertices in $B_1$, so unweighted
level counts fail even the weak degree-one inequality.

\subsection{The diamond}
Let the arrows be $a\to b,a\to c,b\to d,c\to d$.
There is one source and one sink. The first basin and hill boundaries
are empty, and $W=\{a,d\}$. Every nonsink has a conical lower link
with least element $d$; every nonsource has a conical upper link
with greatest element $a$. Thus $S_\F(D)=\{a,d\}$ as well.

\begin{figure}[htbp]
\centering
\begin{tikzpicture}[>=Stealth,scale=1.08,every node/.style={font=\normalsize}]
\node (a) at (0,2) {$a$};\node (b) at (-1,1) {$b$};
\node (c) at (1,1) {$c$};\node (d) at (0,0) {$d$};
\foreach \s/\t in {a/b,a/c,b/d,c/d}\draw[->] (\s)--(\t);
\draw[->,thick] (2,1)--(3.6,1);
\node at (2.8,1.4) {$\Mh$};
\node (aa) at (4.8,2) {$a$};\node (dd) at (4.8,0) {$d$};
\draw[->] (aa)--(dd);
\end{tikzpicture}
\caption{The diamond reduces to an arrow. Both reachability order
complexes are contractible.}
\label{fig:diamond}
\end{figure}
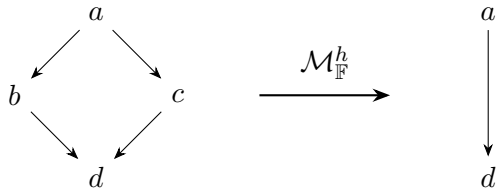

The full order complex consists of the two triangles $abd$ and $acd$
joined along their common edge $ad$. This edge is present in the order
complex because $a$ reaches $d$. The complex is a disc, although the
underlying one-dimensional graph is a circle.
Only $d$ contributes to the descending Morse numbers:
\[
 m_D(t)=P_D(t)=1,\qquad Q_D(t)=0.
\]
All Morse inequalities are equalities.

\subsection{A six-vertex graph with second homology}
Arrange six vertices in three pairs
$\{a_1,a_2\}$, $\{b_1,b_2\}$, $\{c_1,c_2\}$.
Include all four arrows from the $a$-pair to the $b$-pair and all
four from the $b$-pair to the $c$-pair; see Figure~\ref{fig:sphere}.
The descending levels are the $c$-, $b$-, and $a$-pairs, respectively,
and every vertex is retained.

\begin{figure}[htbp]
\centering
\begin{tikzpicture}[>=Stealth,scale=1.08,every node/.style={font=\normalsize}]
\node (a1) at (-1.1,3) {$a_1$};\node (a2) at (1.1,3) {$a_2$};
\node (b1) at (-1.1,1.5) {$b_1$};\node (b2) at (1.1,1.5) {$b_2$};
\node (c1) at (-1.1,0) {$c_1$};\node (c2) at (1.1,0) {$c_2$};
\foreach \s/\t in {a1/b1,a1/b2,a2/b1,a2/b2,b1/c1,b1/c2,b2/c1,b2/c2}
 \draw[->] (\s)--(\t);
\node[anchor=west] at (2,3) {$\mu_2^-(a_i)=1$};
\node[anchor=west] at (2,1.5) {$\mu_1^-(b_i)=1$};
\node[anchor=west] at (2,0) {$\mu_0^-(c_i)=1$};
\end{tikzpicture}
\caption{A DAG whose order complex is the boundary of an octahedron.
The annotations give the single nonzero descending multiplicity at
each vertex.}
\label{fig:sphere}
\end{figure}
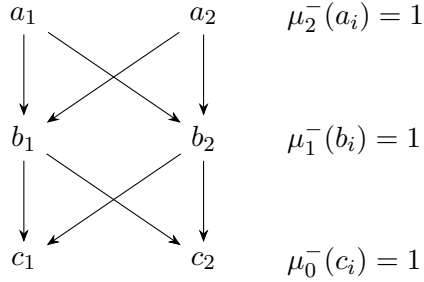

Each $c_i$ is a sink. The lower link of each $b_i$ is the two-point
set $\{c_1,c_2\}$. The lower link of each $a_i$ is the four-cycle
on the $b$- and $c$-pairs. Consequently
\[
 (m_0,m_1,m_2)=(2,2,2).
\]
A chain contains at most one vertex from each pair, and any such
selection is a chain. Therefore
\[
 K(D)=S^0*S^0*S^0\cong S^2,\qquad
 (\beta_0,\beta_1,\beta_2)=(1,0,1).
\]
The strong inequalities are
\[
 2\ge1,\qquad 2-2=0\ge-1,\qquad 2-2+2=2=1-0+1.
\]
The polynomial calculation is
\[
 (2+2t+2t^2)-(1+t^2)=(1+t)^2,
 \qquad Q_D(t)=1+t.
\]

Choose generators $z_1,z_2$ in degree two, $e_1,e_2$ in degree one,
and $c_1,c_2$ in degree zero. One choice of differential is
\[
 dz_1=dz_2=e_2-e_1,\qquad
 de_1=de_2=c_2-c_1.
\]
Both maps have rank one and their composition is zero.
Thus $H_0\cong\F$, $H_1=0$, and
$H_2=\F\langle z_1-z_2\rangle$, giving the stated Betti numbers.

\subsection{Lower links in dimension two}\label{sec:dimtwo}
Suppose $d(D)\le2$. For each nonsink $v$, the lower link $G_v=L^-(v)$
is a graph, possibly with isolated vertices. Write $n_v,e_v,c_v$
for its numbers of vertices, edges, and components.

\begin{proposition}\label{prop:heighttwo}
For $d(D)\le2$,
\[
 m_0=|\Sink(D)|,\qquad
 m_1=\sum_{v\notin\Sink(D)}(c_v-1),\qquad
 m_2=\sum_{v\notin\Sink(D)}(e_v-n_v+c_v),
\]
and $m_r=0$ for $r>2$.
\end{proposition}
\begin{proof}
A nonempty lower link has dimension at most one, so its only
possible reduced homology groups have dimensions $c_v-1$ and
$e_v-n_v+c_v$. The definitions give the formulas.
\end{proof}

The graphs $G_v$ are bipartite. A contribution in degree two means
that $G_v$ contains a cycle; that cycle need not have length four.
For example, orient a six-cycle from one bipartition class to the
other and place a new vertex $v$ above all its vertices. Then
$L^-(v)$ is that six-cycle, so $\mu_2^-(v)=1$, although the link
contains no $K_{2,2}$.

A useful extension of the examples is to add several common maxima.
Let $G$ be any nonempty graph of dimension at most one in the sense
of Proposition~\ref{prop:heightone}, and let $C_a(G)$ be the Hasse
diagram of the poset obtained by placing $a\ge1$ new incomparable
vertices above every vertex of $P(G)$. Let $p,q,e,c$ have the
meaning used there, and put $b=e-p-q+c=\beta_1(G)$.

\begin{proposition}\label{prop:commonmaxima}
The order complex of $C_a(G)$ is $X_a*K(G)$, where $X_a$ is a
discrete set of $a$ vertices. Consequently
\[
 \begin{split}
 m_{C_a(G)}(t)&=q+\bigl(e-p+a(c-1)\bigr)t+ab\,t^2,\\
 P_{C_a(G)}(t)&=1+(a-1)(c-1)t+(a-1)b\,t^2,\\
 Q_{C_a(G)}(t)&=(q-1)+bt .
 \end{split}
\]
\end{proposition}
\begin{proof}
A chain consists of a chain in $P(G)$ and at most one new maximum,
which proves the join description. Every new maximum has lower
link $K(G)$ and contributes $c-1$ in degree one and $b$ in degree
two. The join formula gives the stated Betti numbers; subtraction
gives $Q$.
\end{proof}

This also explains first homology in three-level examples:
disconnected lower graphs can produce $\beta_1>0$ when $a>1$.
The condition is determined here by $c-1$, rather than by the
presence or absence of a particular transitive arrow.

\subsection{Complete layered graphs}\label{sec:layers}
Let $L_1,\ldots,L_k$ be nonempty disjoint layers, ordered from top to
bottom, with $|L_i|=a_i$ and $k\ge2$.
Let $D(a_1,\ldots,a_k)$ have every arrow from $L_i$ to $L_{i+1}$
and no others. Its transitive closure also has all arrows from
$L_i$ to $L_j$ for $i<j$. Both graphs have the same reachability
order. In this subsection $D$ denotes either graph.

\begin{proposition}\label{prop:layers}
The order complex and its Poincar\'e polynomial are
\[
 K(D)=X_{a_1}*\cdots*X_{a_k},\qquad
 P_D(t)=1+\left(\prod_{i=1}^k(a_i-1)\right)t^{k-1}.
\]
For $0\le r\le k-1$,
\[
 m_r(D)=a_{k-r}\prod_{j=k-r+1}^k(a_j-1),
 \qquad
 Q_D(t)=\sum_{r=0}^{k-2}
       \left(\prod_{j=k-r}^k(a_j-1)\right)t^r,
\]
where an empty product is one.
\end{proposition}
\begin{proof}
A chain selects at most one vertex from each layer, proving the
join description. Each discrete factor has reduced homology only
in degree zero, of dimension $a_i-1$. The join formula gives $P_D$.
For $v\in L_i$, the lower link is
$X_{a_{i+1}}*\cdots*X_{a_k}$; its local homology is concentrated in
degree $k-i$, with multiplicity $\prod_{j>i}(a_j-1)$.
Summing over a layer gives $m_r$. Finally write
$a_{k-r}=(a_{k-r}-1)+1$; the two resulting products are precisely
the adjacent coefficients of $(1+t)Q_D$.
\end{proof}

For three layers this becomes
\[
 \begin{split}
 m_D(t)&=c+b(c-1)t+a(b-1)(c-1)t^2,\\
 P_D(t)&=1+(a-1)(b-1)(c-1)t^2,\\
 Q_D(t)&=(c-1)+(b-1)(c-1)t.
 \end{split}
\]
The six-vertex example is $D(2,2,2)$.
More generally, $K(D(2,\ldots,2))$ is the boundary of a
$k$-dimensional cross-polytope, hence a $(k-1)$-sphere, whereas
$m_D(t)=2(1+t+\cdots+t^{k-1})$.
For $D(3,\ldots,3)$, one has $m_r=3\cdot2^r$.
Thus local multiplicities can grow rapidly with dimension.

Reversing the layer order gives
\[
 m_r^+(D)=a_{r+1}\prod_{j=1}^r(a_j-1).
\]
A palindromic layer sequence therefore has equal ascending and
descending Morse numbers. The converse need not hold when singleton
layers make these products vanish: both Morse polynomials equal
$1$ for $D(1,2,3,1)$, although its layer sequence is not palindromic.
All descending inequalities are equalities exactly when $a_k=1$;
otherwise $Q_D\ne0$. This does not mean that every individual
inequality is strict.

\subsection{Local degrees, depth, and coefficient fields}\label{sec:localtopology}
Always $\mu_r^-(v)=0$ for $r>s(v)$, but there is
no general vanishing theorem for the intermediate degrees.

For a small example, take the six vertices
$v,z,b_1,b_2,c_1,c_2$ and arrows
\[
 v\to z,\quad v\to b_2,\quad z\to b_1,\quad
 b_i\to c_j\quad(1\le i,j\le2).
\]
Then $s(v)=3$, but $\mu_2^-(v)=1$. Indeed, the lower link of $v$
consists of the four-cycle on $b_1,b_2,c_1,c_2$ together with the
cone from $z$ on the tree $c_1-b_1-c_2$.
Collapsing that added cone onto its base preserves the four-cycle's
homotopy type. Thus $L^-(v)\simeq S^1$.
Adding one further sink reached directly from $v$ makes its lower
link $S^1\sqcup\{\text{point}\}$ up to homotopy; the same vertex
then contributes in both degrees one and two.

More generally, every finite simplicial homotopy type can occur
as a lower link. Given a nonempty finite simplicial complex $L$,
take the inclusion poset of its nonempty faces and adjoin a greatest
element $v$. The lower link of $v$ is the barycentric subdivision
of $L$, so
\[
 E_r^-(v)\cong\widetilde H_{r-1}(L;\F).
\]
This observation also settles the question of possible torsion:
the face poset of a triangulated real projective plane is a
dimension-two DAG with $H_1(K(D);\mathbb Z)=\mathbb Z/2$.
Its Poincar\'e polynomial is
\[
 P_D(t;\F)=
 \begin{cases}
  1+t+t^2,&\operatorname{char}\F=2,\\
  1,&\operatorname{char}\F\ne2.
 \end{cases}
\]
Here the coefficient field is displayed explicitly to distinguish
the two computations. For the additional greatest vertex,
$E_2^-(v)$ and $E_3^-(v)$ are each one-dimensional in
characteristic two, whereas all its local groups vanish in other
characteristics. Thus $\mathbb F_2$ gives a different answer from
$\mathbb Q$, $\mathbb R$, or $\mathbb F_3$.
These computations follow from the
cellular boundary map given by multiplication by two
\cite[Section~2.2]{Hatcher}. Thus the coefficient field in our
definitions is essential. By contrast, the tree, bipartite, and
complete layered examples above have the same formulas over every
field. No change to the Morse inequalities is needed: one simply
computes both sides over the chosen field.

For the tree $T$ of Section~\ref{sec:tree-example},
$m_T(t)=3+2t$, $P_T(t)=1$, and $Q_T(t)=2$.
A Morse differential can be written
$du_q=f-e$, $du_x=d-e$; its rank is two.
The generator at $x$ is precisely the one missed by the original
Maxwell graph.

\section{Maxwell and path homology: comparisons and questions}\label{sec:outlook}
\subsection{An obstruction to unweighted level counts}
Write $c_j(D)=|B_j(D)|$. The following obstruction applies to any
homology theory with the usual zeroth Betti number.

\begin{proposition}\label{prop:obstruction}
There is no assignment of nonnegative Betti numbers $b_r(D)$, with
$b_0(D)$ equal to the number of weakly connected components, for which
the strong Morse inequalities with $c_j(D)=|B_j(D)|$ hold for every DAG.
\end{proposition}
\begin{proof}
Take the star with arrows $v\to s_i$, $1\le i\le k$, where $k\ge3$.
Its only nonempty descending levels have sizes $c_0=k$ and $c_1=1$.
It is connected, so the degree-one strong inequality would give
\[
 1-k=c_1-c_0\ge b_1-b_0=b_1-1,
 \qquad\text{hence}\qquad b_1\le2-k<0.
\]
This contradicts nonnegativity.
\end{proof}

The star is its own Maxwell graph. Passing to cohomology over a field
leaves its Betti numbers unchanged and does not remove the obstruction.

\subsection{Homology on the unchanged Maxwell graph}
Put $W=W(D)$, $P_W=P(D)|_W$, and define
\begin{equation}\label{eq:Maxwellhomology}
 H_r^{\M}(D;\F)=H_r(\Delta P_W;\F),\qquad
 b_r^{\M}(D)=\dim H_r^{\M}(D;\F).
\end{equation}
Define the local multiplicities in the reachability order of $\M(D)$ by
\[
 \mu_r^{\M}(w)=\dim\widetilde H_{r-1}
       \bigl(\Delta\{u\in W:u<w\};\F\bigr),\qquad
 m_r^{\M}(D)=\sum_{w\in W}\mu_r^{\M}(w).
\]

\begin{proposition}\label{prop:Maxwellhomology}
The numbers $m_r^{\M}(D)$ and $b_r^{\M}(D)$ satisfy the weak and strong
Morse inequalities, including eventual equality. The groups
$H_*^{\M}$ are invariant under the original Maxwell reduction,
under changes of transitive arrows, and under directed subdivision.
The original boundary hierarchies are retained.
\end{proposition}
\begin{proof}
Apply Theorem~\ref{thm:main} to $\M(D)$.
Lemma~\ref{lem:hierarchy}, with $U=W$, gives the same levels and hence
$W(\M(D))=W(D)$. Thus $\M(\M(D))=\M(D)$.
Changing transitive arrows preserves $P_W$.
The topographic argument in Proposition~\ref{prop:subdivision} shows
that subdivision introduces no new essential vertex and leaves the
restricted order on the old essential vertices unchanged.
\end{proof}

For the nine-vertex tree $T$ of Section~\ref{sec:tree-example}, the Morse
and Poincar\'e polynomials are $3+3t$ and $1+t$; their difference is
$2(1+t)$, but $b_1^{\M}(T)=1\ne0=\beta_1(T)$.
Functoriality holds for order-preserving maps carrying essential
vertices to essential vertices.

The original counts suffice if each $w\in B_j$ has lower link in
$P_W$ with the reduced homology of $S^{j-1}$, and all vertices of
$W\setminus\bigcup_j B_j$ have acyclic lower links. Then
$m_j^{\M}=|B_j|$. This holds in the six-vertex example, but not
in the three-leaf star.

\begin{proposition}\label{prop:maxwell-lowdepth}
If $d(D)\le2$, then $S_\F(D)=W(D)$.
Consequently the original Maxwell graph already preserves all
lower and upper local groups and the homology of every order
ideal and filter.
\end{proposition}
\begin{proof}
Let $v\notin W(D)$. It is neither a source nor a sink.
Every strict descendant is a sink and every strict ancestor is
a source: otherwise a path of length three would pass through $v$.
If $v$ reached two sinks, it would be a sink of the first basin
boundary and belong to $B_1$. Thus it has exactly one strict
descendant. Dually it has exactly one strict ancestor.
Both links are points, so $v$ is regular in both directions.
The conclusion follows from Propositions~\ref{prop:preservation}
and~\ref{prop:filtration}.
\end{proof}

For $d(D)\le1$, every vertex is already a source or a sink, so
$W(D)=V(D)$ without any condition on the number of sinks.
In higher dimensions, the tree in Figure~\ref{fig:tree} shows
why the additional homologically critical vertices can be needed.

\subsection{Local Morse numbers for path homology}
Let $A_r(D)$
be the vector space spanned by allowed elementary paths
$[v_0,\ldots,v_r]$, with $v_i\to v_{i+1}$ an arrow for every $i$.
Following \cite{GLMY}, set
\[
 \Omega_r(D)=\{c\in A_r(D):\partial c\in A_{r-1}(D)\},\qquad
 H_r^{\mathrm{path}}(D)=H_r(\Omega_*(D),\partial).
\]
Here $\partial$ is the alternating vertex-deletion boundary and
$\partial_0=0$. In general $\Omega_r(D)$ is a space of linear
combinations; its basis need not consist of elementary paths.

For a vertex $v$, use the \emph{induced} digraphs
\[
 D_{\le v}=D[\{u:u\le v\}],\qquad
 D_{<v}=D[\{u:u<v\}],
\]
with their original arrows, without adding transitive arrows.
The inclusion $D_{<v}\subseteq D_{\le v}$ embeds their path chain
complexes: an allowed chain with allowed boundary in the smaller
graph has the same property in the larger graph.
Define relative local groups and their dimensions by
\begin{equation}\label{eq:pathlocal}
 F_r^-(v)=H_r\!\left(
   \Omega_*(D_{\le v})/\Omega_*(D_{<v})\right),
 \qquad \nu_r^-(v)=\dim F_r^-(v).
\end{equation}
Write $p_r(D)=\sum_v\nu_r^-(v)$ and
$b_r^{\mathrm{path}}(D)=\dim H_r^{\mathrm{path}}(D)$.
The corresponding polynomials are $p_D(t)=\sum_r p_r(D)t^r$ and
$P_D^{\mathrm{path}}(t)=\sum_r b_r^{\mathrm{path}}(D)t^r$.

\begin{proposition}[Path-homological Morse inequalities]\label{prop:pathMorse}
For every finite DAG and every $q\ge0$,
\begin{equation}\label{eq:pathMorse}
 \sum_{r=0}^q(-1)^{q-r}p_r(D)
 \ge\sum_{r=0}^q(-1)^{q-r}b_r^{\mathrm{path}}(D),
\end{equation}
with equality for all sufficiently large $q$. Also
$p_r(D)\ge b_r^{\mathrm{path}}(D)$.
If $D$ is transitive, then $F_r^-(v)\cong E_r^-(v)$.
\end{proposition}
\begin{proof}
Choose a linear extension with descendants first. For its first $i$
vertices put $D_i=D[\{v_1,\ldots,v_i\}]$.
The inclusions identify the quotient chain complexes:
\begin{equation}\label{eq:pathquotient}
 \Omega_*(D_{\le v_i})/\Omega_*(D_{<v_i})
 \ \cong\ \Omega_*(D_i)/\Omega_*(D_{i-1}).
\end{equation}
Indeed, every new allowed path starts at $v_i$, and all its remaining
vertices are strict descendants. The condition that a boundary be
allowed also separates according to the initial vertex: deleting an
endpoint always leaves an allowed path, whereas every forbidden face
comes from an interior deletion and retains the initial vertex.
Consequently forbidden faces cannot cancel between chains with
different initial vertices, so each initial-vertex component of an
$\Omega$-chain is itself an $\Omega$-chain. It follows that
\[
 \Omega_*(D_i)=\Omega_*(D_{i-1})+\Omega_*(D_{\le v_i}),\qquad
 \Omega_*(D_{i-1})\cap\Omega_*(D_{\le v_i})=\Omega_*(D_{<v_i}).
\]
The subcomplex identities give the chain isomorphism \eqref{eq:pathquotient}.

The exact sequences of this filtration give
$p_r-b_r^{\mathrm{path}}=A_r+A_{r+1}$, where $A_r\ge0$ and $A_0=0$,
as in Theorem~\ref{thm:main}. Bounded path lengths give eventual
equality. For a transitive DAG, $\Omega_*(D)$ is the simplicial chain
complex of $K(D)$, and the local quotient is the cone--base complex
of Lemma~\ref{lem:attachment}.
\end{proof}

The local groups are independent of the linear extension.
Reversing arrows gives $F_r^+(v)$ and the ascending inequalities.

Grouping paths by their initial vertex makes the differential
triangular, with diagonal homology $F_*^-(v)$.
The cancellation proof of Proposition~\ref{prop:chain}, or algebraic
Morse reduction \cite{Skoldberg}, therefore produces
\[
 \mathcal P_r(D)=\bigoplus_{v\in V(D)}F_r^-(v),\qquad
 H_r(\mathcal P_*(D))\cong H_r^{\mathrm{path}}(D).
\]
Its differential lowers the reachability label. Set
\[
 S_{\mathrm{path}}=W(D)\cup
 \{v:F_*^-(v)\ne0\}\cup\{v:F_*^+(v)\ne0\}.
\]
By Lemma~\ref{lem:hierarchy}, the Hasse diagram on $S_{\mathrm{path}}$
retains both hierarchies. The chain model carries extra data:
its differential depends on choices and may connect nonadjacent
comparable vertices. It need not be the ordinary path complex
of this Hasse diagram.

\subsection{Agreement on complete layered graphs}\label{sec:agreement}
The inclusion of the original arrows into the transitive closure
gives a chain map
\[
 \Omega_*(D)\longrightarrow C_*(K(D);\F).
\]
Its relative versions at a vertex induce maps
$\gamma_{v,r}:F_r^-(v)\to E_r^-(v)$.

\begin{lemma}\label{lem:comparison}
If every $\gamma_{v,r}$ is an isomorphism, then the inclusion above
induces an isomorphism on homology and $p_r(D)=m_r(D)$ for every $r$.
\end{lemma}
\begin{proof}
Filter both complexes by a linear extension with descendants first.
Their successive relative homology groups are $F_r^-(v)$ and
$E_r^-(v)$. Induction using the long exact sequences of the pairs
proves the assertion for each filtration step.
\end{proof}

Transitivity is a sufficient condition, but is not necessary.

\begin{proposition}\label{prop:layer-agreement}
For $D=D(a_1,\ldots,a_k)$ with only consecutive-layer arrows,
and also for its transitive closure,
\[
 H_r^{\mathrm{path}}(D)\cong H_r(K(D);\F)\cong H_r^\M(D;\F),
 \qquad p_r(D)=m_r(D)=m_r^\M(D).
\]
The common Morse numbers are those of Proposition~\ref{prop:layers}.
\end{proposition}
\begin{proof}
First consider the Maxwell graph. If no layer is a singleton,
the descending hierarchy visits every layer, so $W=V$.
Otherwise let $\ell$ and $u$ be the first and last singleton layers,
numbered from top to bottom. The ascending hierarchy visits layers
$1,\ldots,\ell$ and then stops; the descending hierarchy visits
layers $k,\ldots,u$ and then stops. Therefore
\[
 W=\bigcup_{i\le\ell}L_i\ \cup\ \bigcup_{i\ge u}L_i.
\]
An omitted vertex lies strictly between these two singleton layers.
Each of its order links has a singleton join factor and is a cone.
Thus $S_\F=W$, and Proposition~\ref{prop:preservation} proves the
Maxwell assertions. In particular, equality of the invariants does
not require $W=V$.

For the path assertion, put $U_j=\F[L_j]$ and
$Z_j=\ker(\epsilon_j:U_j\to\F)$, where $\epsilon_j$ sums coefficients.
Fix $v\in L_i$, and let $s=k-i$.
The quotient
$Q_*(v)=\Omega_*(D_{\le v})/\Omega_*(D_{<v})$ consists of
chains starting at $v$. For $1\le r\le s$,
\[
 Q_0(v)=\F[v],\qquad
 Q_r(v)\cong
 Z_{i+1}\otimes\cdots\otimes Z_{i+r-1}\otimes U_{i+r}.
\]
Indeed, an interior deletion skips a layer and gives a forbidden
face. Its coefficients must sum to zero in that layer.
Different skipped layers give disjoint types of forbidden faces.
These conditions put every interior tensor factor in $Z_j$;
the last factor remains unrestricted.

In the quotient, the first face vanishes, the interior faces
cancel, and the last face gives the differential. Up to sign,
this is $\epsilon_{i+r}$ on the last factor, followed by the
inclusion $Z_{i+r-1}\subset U_{i+r-1}$ when $r>1$.
Since each $\epsilon_j$ is surjective,
\[
 H_r(Q_*(v))=
 \begin{cases}
 Z_{i+1}\otimes\cdots\otimes Z_k,&r=s,\\
 0,&r\ne s.
 \end{cases}
\]
For $s=0$ the tensor product is $\F$ in degree zero.
Under inclusion into the order complex, these top cycles are the
join cycles of the discrete lower layers. They form the usual
tensor-product basis of $E_s^-(v)$ from the join formula.
Hence every $\gamma_{v,r}$ is an isomorphism.
Lemma~\ref{lem:comparison} proves the path assertions.
The transitive closure follows directly from
Proposition~\ref{prop:pathMorse}.
\end{proof}

For three layers the path Morse numbers are therefore
\[
 (p_0,p_1,p_2)=\bigl(c,\ b(c-1),\ a(b-1)(c-1)\bigr).
\]
The linear combinations in $\Omega_2$ are essential here.
Although $[a_i,b_j,c_l]$ is not an individual allowed-boundary
chain without the shortcut $a_i\to c_l$, the difference
$[a_i,b_j,c_l]-[a_i,b_{j'},c_l]$ is.
Ignoring such cancellations would incorrectly remove all
degree-two local groups.

\subsection{Distinguishing examples and admissible deletions}\label{sec:differences}
The diamond provides a simple obstruction to preserving path homology
by the original Maxwell reduction. Subdivide its arrow $a\to b$
to obtain the graph $J$:
\[
 a\to x\to b\to d,\qquad a\to c\to d.
\]
Both original Maxwell graphs are the single arrow $a\to d$.
For the diamond, $\Omega_2$ is spanned by
$[a,b,d]-[a,c,d]$, whose boundary fills its only undirected cycle.
Thus its first path Betti number is zero.
In $J$, the three allowed two-paths have distinct forbidden shortcut
faces $[a,b]$, $[x,d]$, and $[a,d]$. Hence $\Omega_2(J)=0$.
Its five arrows and five vertices form a connected cycle, so
\[
 H_1^{\mathrm{path}}(J;\F)\cong\F.
\]
Thus the original Maxwell graph does not determine path homology,
which is also not subdivision-invariant. Since $J$ has a unique sink,
all nonsink lower order links are cones, so $m_1(J)=0$.
This also proves the failure of substituting path Betti numbers
in \eqref{eq:strong} noted in the introduction.

In $J$, only $d$ contributes
in degree zero and only $a$ contributes in degree one:
\[
 p_J(t)=1+t=P_J^{\mathrm{path}}(t).
\]
Removing $a$ leaves a directed tree, so its relative group records
the cycle; all other nonsink local groups vanish. Reversing arrows
gives $S_{\mathrm{path}}=\{a,d\}$. The chain model has one generator
in each of degrees zero and one and zero differential. The arrow
$a\to d$ itself has no degree-one path homology.

\paragraph{A two-dimensional class lost by the original reduction.}
Let $\Sigma$ be obtained from the tree $T$ in
Section~\ref{sec:tree-example} by replacing $x$ with two incomparable
vertices $x,y$, each having the former predecessors and successors
of $x$. Explicitly, its ten vertices have arrows
\[
 a\to p,\ b\to p,\qquad
 p,c\to x,y,\qquad x,y\to q,d,\qquad q\to e,f,
\]
where an arrow between two lists denotes all the indicated arrows.
Its order complex splits as a join
\[
 K(\Sigma)=K(\{a,b,p,c\})*\{x,y\}*K(\{q,d,e,f\}).
\]
The first and last factors each have two contractible components,
so $K(\Sigma)\simeq S^0*S^0*S^0\cong S^2$.
The only nonempty boundary levels are
\[
 B_0=\{d,e,f\},\quad B_1=\{q\},\qquad
 H_0=\{a,b,c\},\quad H_1=\{p\}.
\]
The original Maxwell graph deletes both $x$ and $y$. Its order
complex is the join of the first and last factors above and is
homotopy equivalent to $S^1$. Thus this reduction loses a genuine
two-dimensional class and introduces a one-dimensional class.
Both $x$ and $y$ have nonzero lower and upper local homology, so
our homological Maxwell graph retains them.

The descending local contributions for $\Sigma$ are one in degree
zero at each of $d,e,f$, one in degree one at each of $q,x,y$,
and one in degree two at each of $p,c$. Therefore
\[
 m_\Sigma(t)=3+3t+2t^2,\qquad
 P_\Sigma(t)=1+t^2,\qquad Q_\Sigma(t)=2+t.
\]
On the original Maxwell graph the corresponding polynomials are
$3+3t$, $1+t$, and $2$.

\paragraph{The same two-dimensional class is invisible to path homology.}
Subdivide the arrow $x\to d$ in $\Sigma$ by a new vertex $z$;
call the resulting graph $\Sigma'$ (Figure~\ref{fig:nontrivial}).
Proposition~\ref{prop:subdivision} shows that its reachability
homology, local Morse numbers, and Maxwell graph are unchanged.
In particular, $K(\Sigma')\simeq S^2$. Here the homotopy statement
also follows directly: the new vertex has the single strict
descendant $d$, so deleting it collapses its conical star onto
its link and recovers $K(\Sigma)$.

The path computation is quite different. For $\Sigma$, a basis
of $\Omega_2$ consists of the four chains
\[
 [u,x,w]-[u,y,w],\qquad u\in\{p,c\},\quad w\in\{q,d\}.
\]
Their boundaries have rank three; the alternating sum of these
four chains generates $H_2^{\mathrm{path}}(\Sigma;\F)$.
At each of $p,c$, the two relative boundaries have rank one,
giving one local degree-two generator; the other local path
groups agree with the reachability groups listed above.
For $\Sigma'$, only the two chains with $w=q$ remain in
$\Omega_2$, and their boundaries are independent. To verify that
these lists are complete, inspect the missing shortcut for each
allowed two-path. All other shortcuts have a unique allowed
filling, so their coefficients must vanish. Similarly,
$\Omega_3=0$ in both graphs: every allowed three-path has an
interior face with a unique allowed filling. The same argument
excludes higher path chains.

The graph $\Sigma'$ is connected, with eleven vertices and
thirteen arrows. Consequently
\[
 \beta_1^{\mathrm{path}}(\Sigma')=13-10-2=1,
 \qquad H_r^{\mathrm{path}}(\Sigma';\F)=0\quad(r\ge2).
\]
Its local path groups occur only at the sinks $d,e,f$ in degree
zero and at $q,x,y$ in degree one. At $p$ and $c$, the relative
boundary of $[u,x,q]-[u,y,q]$ cancels the potential degree-one
generator, leaving no local homology. Thus
\[
 p_{\Sigma'}(t)=3+3t,\qquad
 P_{\Sigma'}^{\mathrm{path}}(t)=1+t.
\]
Our theory therefore retains a nonzero $H_2$ that neither of the
other two constructions records for this graph. The path result
is a different invariant of the arrows, whereas the loss under
the original Maxwell reduction changes the reachability homology
itself.

\begin{figure}[t]
\centering
\begin{minipage}{.47\textwidth}
\centering
\begin{tikzpicture}[scale=.83,>=Stealth,every node/.style={inner sep=2pt}]
 \node (a) at (-1.8,4) {$a$}; \node (b) at (-.2,4) {$b$};
 \node (p) at (-1,3) {$p$}; \node (c) at (1.2,3) {$c$};
 \node (x) at (-1,2) {$x$}; \node (y) at (1.2,2) {$y$};
 \node (q) at (-1,1) {$q$}; \node (z) at (.5,.8) {$z$};
 \node (e) at (-1.8,0) {$e$}; \node (f) at (-.2,0) {$f$};
 \node (d) at (1.8,0) {$d$};
 \foreach \u/\v in {a/p,b/p,p/x,p/y,c/x,c/y,x/q,y/q,x/z,z/d,y/d,q/e,q/f}
  \draw[->] (\u)--(\v);
\end{tikzpicture}
\par (a) $\Sigma'$: a two-dimensional class.
\end{minipage}\hfill
\begin{minipage}{.50\textwidth}
\centering
\begin{tikzpicture}[scale=.83,>=Stealth,every node/.style={inner sep=2pt}]
 \node (o) at (0,4) {$o$}; \node (a) at (0,3) {$a$};
 \node (b) at (0,2) {$b$}; \node (c) at (0,1) {$c$};
 \node (d) at (0,0) {$d$}; \node (e) at (2,1) {$e$};
 \node (u) at (2,3) {$u$};
 \foreach \v/\w in {o/a,a/b,b/c,c/d,b/e,u/e,u/d}
  \draw[->] (\v)--(\w);
 \draw[->] (a) to[bend right=55] (d);
 \draw[->] (o) to[bend right=30] (b);
 \draw[->] (o) to[bend right=46] (c);
 \draw[->] (o) to[bend right=62] (d);
\end{tikzpicture}
\par (b) $\Lambda$: sharper inequalities.
\end{minipage}
\caption{Two connected examples with nonzero reduced reachability
homology. In (a), replacing $x\to z\to d$ by $x\to d$ gives
$\Sigma$.}
\label{fig:nontrivial}
\end{figure}
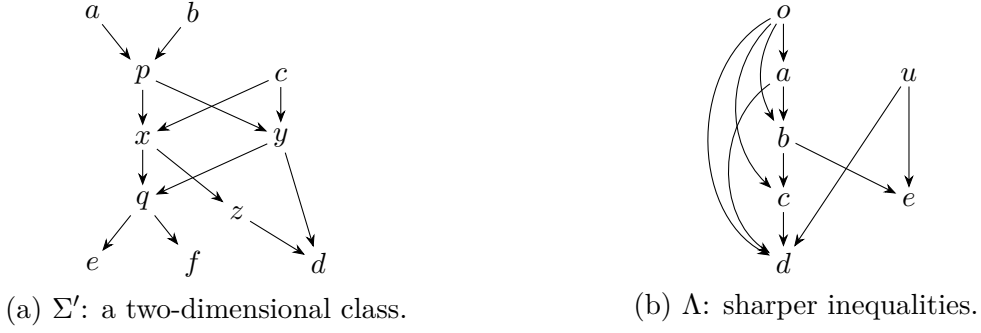

\paragraph{Sharper inequalities with nonzero first homology.}
Let $\Lambda$ have seven vertices $o,u,a,b,c,d,e$ and arrows
\[
 a\to b\to c\to d,\quad a\to d,\quad
 o\to a,b,c,d,\quad b\to e,\quad u\to d,e.
\]
The order complex on all vertices except $u$ is a cone with
greatest vertex $o$. Adding $u$ attaches a cone on the two-point
set $\{d,e\}$, so $K(\Lambda)\simeq S^1$.
Only $d,e$ contribute in degree zero and $b,u$ in degree one:
the other nonsink lower links are cones. Hence
\[
 m_\Lambda(t)=2+2t,\qquad P_\Lambda(t)=1+t,
 \qquad Q_\Lambda(t)=1.
\]
The Maxwell set is $\{o,u,b,d,e\}$. Deleting $a,c$ leaves the
same local Morse numbers and homology, so its three polynomials
are identical to these.
These reachability counts are minimal for a connected DAG with
$\beta_1=1$: a unique sink would make its order complex a cone,
so $m_0\ge2$, and the strong inequality gives $m_1\ge m_0\ge2$.

For path homology, the only allowed two-chains with allowed
boundary are spanned by
\[
 [o,a,b],\quad[o,b,c],\quad[o,c,d],\quad[o,a,d].
\]
Their boundaries are independent, and there are no nonzero path
chains of degree at least three. The latter assertion follows
because every allowed three-path has a forbidden interior face
with a unique filling; the same holds in degree four.
Since $\Lambda$ has eleven arrows, this gives
$\beta_1^{\mathrm{path}}=11-6-4=1$ and no higher path homology.
At $o$, however, the relative boundaries of these four two-chains
have rank three: they are the incidence vectors of the cycle
$a-b-c-d-a$. Thus the local path group at $o$ is one-dimensional
in degree two. In addition to $b,u$, the vertex $a$ contributes
one local path generator in degree one. It follows that
\[
 p_\Lambda(t)=2+3t+t^2,\qquad
 P_\Lambda^{\mathrm{path}}(t)=1+t,\qquad
 Q_\Lambda^{\mathrm{path}}(t)=1+t.
\]
All three homologies agree, but our degree-one strong inequality
is an equality, whereas the local path inequality has defect one:
\[
 (2-2)=(1-1)=0,\qquad (3-2)>(1-1)=0.
\]
The extra degree-two and degree-one path generators cancel in
the path Morse complex. The comparison concerns the canonical
local constructions, not an optimal choice of discrete gradient.

\begin{center}
\small
\setlength{\tabcolsep}{6pt}
\begin{tabular}{llccc}
\toprule
Graph & Construction & Poincar\'e $P(t)$ & Morse $m(t)$ & Defect $Q(t)$\\
\midrule
$\Sigma$ & Reachability & $1+t^2$ & $3+3t+2t^2$ & $2+t$\\
 & Maxwell & $1+t$ & $3+3t$ & $2$\\
 & Path & $1+t^2$ & $3+3t+2t^2$ & $2+t$\\
\midrule
$\Sigma'$ & Reachability & $1+t^2$ & $3+3t+2t^2$ & $2+t$\\
 & Maxwell & $1+t$ & $3+3t$ & $2$\\
 & Path & $1+t$ & $3+3t$ & $2$\\
\midrule
$\Lambda$ & Reachability & $1+t$ & $2+2t$ & $1$\\
 & Maxwell & $1+t$ & $2+2t$ & $1$\\
 & Path & $1+t$ & $2+3t+t^2$ & $1+t$\\
\bottomrule
\end{tabular}
\end{center}
Here $m(t)$ denotes the local Morse polynomial of the indicated
construction (written $p_D(t)$ for paths), and $Q(t)$ is computed
from that construction's own $P(t)$. Every row has nonzero reduced
homology, and all computations hold over any field. For $\Sigma$
and $\Sigma'$, preserving the two-dimensional class is the relevant
distinction; defects attached to different homologies do not rank
the theories. For $\Lambda$, the homologies coincide, so the defect
polynomials give a direct comparison of sharpness.

A useful deletion is possible when $z$ has exactly one incoming arrow
$u\to z$, one outgoing arrow $z\to w$, and $u\to w$ is already present.
The retraction $r:D\to D-z$ sending $z$ to $w$ and fixing other
vertices satisfies $\mathrm{id}_D\simeq i\circ r$ via $z\to w$,
where $i$ is inclusion. Thus deletion preserves path homology;
see \cite[Section~6]{WangRen}. It preserves reachability on the
remaining vertices, and $z$ belongs to no boundary level by the
argument of Proposition~\ref{prop:subdivision}. Both hierarchies
therefore survive. Inserting a missing shortcut first is not allowed.

\subsection{Comparison with discrete gradients}\label{sec:gradient-comparison}
The construction in \cite[Definition~2.1]{LinWangYau} starts with
$f:V(D)\to[0,\infty)$ and assigns
$f([v_0,\ldots,v_r])=\sum_i f(v_i)$.
Each path may have at most one immediate extension of the same
value and at most one immediate face of the same value.
Pairing uses the insertion or deletion of a zero-valued vertex.
The corresponding gradient $V$ defines the chain flow
$\Phi=I+\partial V+V\partial$.
When the function extends to the transitive closure and
$V\Omega_*(D)\subseteq\Omega_*(D)$, stabilized critical paths give
the reduced complex described in \cite[Theorem~2.17]{LinWangYau}.

Our reduction in Proposition~\ref{prop:chain} instead cancels
acyclic pairs inside the blocks labelled by a vertex, leaving its
local homology as the corresponding graded summand. The path
version does the same with $F_*^-(v)$.
Both methods transfer a differential to a smaller chain space,
but neither that differential nor its coefficients are generally
the arrows of a reduced graph. Our graded dimensions are fixed
before the cancellation choices are made. A gradient construction
can change the numbers of critical paths when its auxiliary
function changes.

The permitted vertex functions should not be confused with all
Forman matchings on the order complex. This distinction is already
visible on the directed $K_{2,2}$, which is transitive.

\begin{example}\label{ex:LWY}
Write its arrows as $a_i\to b_j$, $1\le i,j\le2$.
A function of the preceding kind has at most one zero vertex:
two zeros in the same part give two equal-value extensions of a
vertex in the other part, and two zeros in opposite parts give
two equal-value faces of their connecting edge.
With no zeros, all four vertices and all four edges are critical.
With one zero, exactly two vertices and two edges are critical.
For instance, if $f(a_1)=0$ and the other values are positive,
the critical vertices are $a_1,a_2$ and the critical edges are
$a_2b_1,a_2b_2$. Thus the possible critical counts are $(4,4)$
and $(2,2)$, whereas the Betti numbers are $(1,1)$.

As an undirected cell complex, the four-cycle does admit a Forman
matching with one critical vertex and one critical edge: pair the
other three vertices with the edges leading towards a root in a
spanning tree. The fourth edge is critical, and the tree gives no
closed gradient path. This matching does not come from the
vertex-additive functions above.
\end{example}

Consequently, the existence of a small Forman Morse complex does
not by itself imply that the digraph construction of
\cite{LinWangYau} attains the Betti numbers. Likewise, our canonical
multiplicities should not be interpreted as optimal critical counts
over all possible matchings.

\subsection{Further questions}\label{sec:questions}
The examples suggest three questions.

\begin{enumerate}[label=(\arabic*),leftmargin=*]
\item Find structural conditions on the arrows, weaker than
transitivity, that make all local comparison maps
$\gamma_{v,r}$ isomorphisms. Complete layered graphs satisfy this
condition, while the subdivided diamond does not.
Which graph operations preserve it?

\item Develop deletion rules that preserve path homology together
with both Maxwell hierarchies. The admissible deletion above gives
one such rule, but the vertex set $S_{\mathrm{path}}$ alone does not
determine a suitable graph: its transferred differential contains
additional information. When can that information be represented
by the ordinary path complex of a smaller DAG?

\item For fixed numbers of vertices and fixed dimension, determine
the attainable Morse and Betti polynomials and their sharp bounds.
Section~\ref{sec:dimone} settles the unrestricted dimension-one
realizability question. In higher dimensions the problem must
allow local homology in several degrees and dependence on the
coefficient field.
\end{enumerate}

\bigskip
\noindent\textbf{AI use statement.} LLMs were used for proof refinement,
literature review, and manuscript revision.

\bigskip
\noindent The University of Texas Rio Grande Valley, Brownsville, Texas, USA\\
\href{mailto:oleg.musin@utrgv.edu}{\texttt{oleg.musin@utrgv.edu}}
\end{document}